\documentclass[12pt]{article}
\usepackage{amssymb,amsmath,amsfonts}
\usepackage{amsthm}
\usepackage{color}
\numberwithin{equation}{section} \theoremstyle{plain}
\newtheorem{theorem}{Theorem}[section]

\newtheorem{lemma}[theorem]{Lemma}
\theoremstyle{definition}
\newtheorem{definition}{Definition}[section]
\newtheorem{example}{Example}[section]

\theoremstyle{remark}
\newtheorem{remark}{\rm\bf Remark}[section]

\begin{document}

\title{On Lin's condition for products of random variables with joint singular distribution}\author{Alexander Il'inskii$^{1}$ and  Sofiya  Ostrovska$^{2}$}
\date{}

\maketitle

\begin{center}
 {\it $^1$Department of Mathematics and Informatics, Karazin National University, Kharkov, Ukraine\\
 E-mail: iljinskii@univer.kharkov.ua }
\end{center}
\begin{center}
 {\it $^2$Department of  Mathematics, Atilim University, Ankara, TURKEY\\
E-mail: sofia.ostrovska@atilim.edu.tr}\\
\end{center}

\begin{abstract} Lin's condition is used to establish the moment determinacy/indeterminacy of absolutely continuous probability distributions. Recently, a number of papers related to Lin's condition for functions of random variables have emerged. In this work, Lin's condition is studied for the product of random variables with given densities in the case when their joint distribution is singular.\end{abstract}

\noindent {\small \textbf{Keywords:} random variable, singular  distribution, Lin's condition
\vspace{0.2cm}}

\noindent {\small \textbf{Mathematics Subject Classifications:}
60E05}\vspace{0.2cm}

\section{Introduction}

Lin's condition plays a significant role in the investigation of the moment determinacy of absolutely continuous probability distributions. Generally speaking, it is used along with Krein's logarithmic integral to establish so-called ``converse criteria". See, for example, Theorems 5--10 in \cite[Section 5]{recent}.
This condition was introduced and applied by G. D. Lin \cite{lin}, while the name `Lin's condition' was put forth in
\cite{bernoulli}. Let us recall the pertinent notions.

\begin{definition} Let $f$ be a probability density continuously differentiable on $(0,\infty)$. The function
\begin{equation} \label{linfun}
L_f(x):=-\frac{xf^\prime(x)}{f(x)}
\end{equation}
is called \textit{Lin's function} of $f$.
\end{definition}

Formula \eqref{linfun} implies that Lin's function of $f$ is defined only at the points where $f$ does not vanish. In this article, we consider   probability densities of positive random variables whose Lin's functions are defined for all $x>0.$

\begin{definition}\label{lincond} Let $f\in C^1(0,\infty)$ be a probability density of a positive random variable. It is said that $f$ satisfies \textit{ Lin's condition} on $(x_0,\infty)$ if $L_f(x)$ is monotone increasing on $(x_0,\infty)$ and $\displaystyle \lim_{x\rightarrow +\infty}L_f(x)=+\infty$.
\end{definition}

 Due to the importance of this condition for establishing the moment (in)determinacy of absolutely continuous probability distributions (see, \cite{recent, stirzaker} and references therein), it has  to be examined how operations on random variables impact Lin's condition. Recently, Kopanov and Stoyanov in \cite{kopanov} proved that if a density $f$ of a positive random variable $X$ satisfies Lin's condition, then the densities of $X^r, r>0$ and $\ln X$  satisfy this condition, too. In addition, if $L_f(x)/x\rightarrow +\infty$ as $x\rightarrow +\infty,$ then the density of $e^X$ also satisfies Lin's condition as well.

The moment problem for products of random variables leads naturally to the
question concerning Lin's condition for the densities of products. This question was placed in
\cite{kopanov}, where it was inquired whether Lin's condition is inherited by the products of the random variables whose densities satisfy Lin's condition. It was also conjectured that the answer is affirmative   in the case of an absolutely continuous joint distribution. Although this is valid for independent random variables, the results of \cite{arxiv} show that, in general, this assertion may not be true. For more information  on the moment problem for products we refer to \cite{linst1, linst2, counter} and \cite[Section 6]{recent}.

In the present paper, this problem is investigated for the case when joint distribution of factors is singular rather than absolutely continuous. The main result of this work is the next statement.

\begin{theorem}\label{th1} Let  $f_1$ and $f_2$  be  densities of  positive random variables, both satisfying Lin's condition on $(0,+\infty)$.
Then, there exists a random vector $(\xi_1,\xi_2)$ possessing a singular distribution $P$ and satisfying the following conditions:
\begin{enumerate} \item $\xi_1$ and $\xi_2$ have densities $f_1$ and $f_2$ respectively;

\item the density $g$ of the product $\xi_1\cdot\xi_2$ is continuously differentiable on $(0,+\infty);$

\item the following equalities hold: \begin{equation}\label{limsupinf}
\limsup_{x\rightarrow +\infty} L_g(x)=+\infty\,,\quad\liminf_{x\rightarrow +\infty} L_g(x)=-\infty\,.
\end{equation}
\end{enumerate}
\end{theorem}

Obviously, equalities \eqref{limsupinf} imply that $g$ does not satisfy Lin's condition on any interval $(x_0,\infty)$.

\begin{remark} Assuming that  the densities of random variables of $\xi_1$ and $\xi_2$ satisfy Lin's condition on $(0,+\infty),$ we deduce that, in general, this property is not inherited by their product $\xi_1\xi_2$ in the case of a singular joint distribution of $\xi_1$ and $\xi_2.$\end{remark}

\section{Proof of the main theorem}

We start with the construction of a singular distribution with given marginal positive densities, which will be used in the sequel.

\begin{lemma}\label{lem1} Let $f_1$ and $f_2$ be continuous probability densities positive on $(0,+\infty)$ and vanishing elsewhere. Then, there exists a singular distribution $\tilde{P}$ concentrated on a curve $x_2=\varphi(x_1)$, where $\varphi$  is a continuously differentiable  strictly increasing function, and so that the projections of $\tilde{P}$ on the coordinate axes have the given densities $f_1(x_1)$ and $f_2(x_2)$.\end{lemma}

\begin{proof} To begin with, consider the set of points $\{u_{ij}:i\in \mathbb{N}_0, 0\leqslant j\leqslant 2^i\}$ defined by the conditions:
\begin{equation}\label{upoints}
u_{i0}=0,\quad\int_{u_{ij}}^{u_{i,j+1}}f_1(x_1)dx_1=\frac{1}{2^i}\;\;(j=0,1,\dots ,2^i-1),\quad u_{i,2^i}=+\infty\,.
\end{equation}
Clearly \eqref{upoints} defines all of the points $u_{ij}$ uniquely. Besides, for all $i\in \mathbb{N}_0,$ one has:
\begin{equation}\label{ineq} 0=u_{i0}<u_{i1}<\dots <u_{i, 2^i}=+\infty\;\;\mathrm{and}\;\; u_{ij}=u_{i+1, 2j}, \;0\leq j\leq 2^i.\end{equation}
Likewise, let  $\{v_{ij}:i\in \mathbb{N}_0, 0\leqslant j\leqslant 2^i\}$ be the set of points specified by  the conditions:
\begin{equation*}\label{vpoints}
v_{i0}=0,\quad\int_{v_{ij}}^{v_{i,j+1}}f_2(x_2)dx_2=\frac{1}{2^i}\;\;(j=0,1,\dots ,2^i-1),\quad v_{i,2^i}=+\infty\,.
\end{equation*}
Conditions similar to \eqref{ineq} are satisfied for these points, too.

For each $m\in \mathbb{N},$ consider the set
\begin{equation*}\label{km}
K_m:=\bigcup_{j=0}^{2^m-1}(u_{mj},u_{m,j+1})\times
(v_{mj},v_{m,j+1})\subset \mathbb{R}^2
\end{equation*}
and the two-dimensional probability distribution $P_m$ whose density is given by:
\begin{equation*}\label{denfm}f_m(x_1,x_2)=\left\{\begin{array}{cl}2^m f_1(x_1)f_2(x_2)&\text{when}\;\;(x_1,x_2)\in K_m,\\
0&\text{elsewhere}.\end{array}\right.\end{equation*}
The projections of $P_m$ on the $x_1$- and $x_2$-axes have densities $f_1(x_1)$ and $f_2(x_2)$, respectively. Indeed, for any given $m\in \mathbb{N},$ opt for $x_1^*\in (u_{mj}, u_{m,j+1})$ and find $$\int_{-\infty}^{+\infty} f_m(x_1^*,x_2)dx_2=
2^m f_1(x_1^*)\int_{v_{mj}}^{v_{m,j+1}}f_2(x_2)dx_2=f_1(x_1^*).$$ In the same manner, it can be established that the projection of $P_m$ on the $x_2$-axis has density $f_2.$
The sequence $P_m$ of probability distributions is weakly convergent to a two-di\-men\-sional distribution $\tilde{P}$ whose support $\displaystyle K=\cap_{m=1}^\infty K_m $ has zero Lebesgue's measure in $\mathbb{R}^2.$ What is more, $K$ is a continuous curve passing through all points $(u_{ij}, v_{ij}).$

Next, let us derive an equation of $x_2=\varphi(x_1)$ of the curve $K.$ Select a point $(a,b)\in K$ and consider the sequence of rectangles $I_m=J_m\times H_m,$ where $J_m=(u_{mj},u_{m,j+1}), H_m=(v_{mj}, v_{m,j+1})\;\text{with}\;j\; \text{such\; that}\; (a,b)\in I_m.$ Directly from  the construction of sets $K_m,$ one has:
$$
\int_{J_m}f_1(x_1)dx_1=\int_{H_m}f_2(x_2)dx_2,
$$
whence, by the Mean Value Theorem:
$$
f_1(x_1^*)|J_m|=f_2(x_2^*)|H_m|\;\text{for\;some}\;x_1^*\in J_m,\;x_2^*\in H_m.
$$
Consequently,
$$
\frac{|H_m|}{|J_m|}=\frac{f_1(x_1^*)}{f_2(x_2^*)}.
$$
Since all rectangles $I_m$ contain $(a,b),$ taking $m\to\infty,$ implies that $x_1^*\to a, x_2^*\to b$ and the ratio $\displaystyle\frac{|H_m|}{|J_m|}$ approaches the slope $k$ of the tangent line to $K$ at $(a,b).$  Thence, after passing to limit as $m\to\infty,$ one obtains:
$$k=\frac{f_1(a)}{f_2(b)}$$ or, equivalently, $$\varphi^\prime(a)=\frac{f_1(a)}{f_2(\varphi(a))}. $$ Due to the arbitrary selection of $a$, it follows that $x_2=\varphi(x_1)$ satisfies the initial value problem below:
\begin{equation}
\label{ivp}(x_2)^\prime=\frac{f_1(x_1)}{f_2(x_2)},\quad x_2(0)=0.
\end{equation}
Since $\varphi^\prime(x_1)>0$ for all $x_1\in (0,\infty),$ it follows that $\varphi$ is strictly increasing.
\end{proof}

\begin{example}%\label{ex1}
If $f_1=f_2$, that is random variables $\xi_1$ and $\xi_2$ are identically distributed, then $K$ is the straight line $x_2=x_1.$\end{example}

\begin{lemma}\label{lem2}Let $(\xi_1,\xi_2)$ be a random vector with distribution $\tilde{P}$ constructed  above. Then, the density
$g(z,\tilde{P})$ of the product $\xi_1\cdot \xi_2$ is given by:
\begin{equation}\label{pz}
g(z,\tilde{P})=f_1(x_1(z))\frac{x_1(z)}{z+\varphi^\prime(x_1(z))x_1^2(z)}.
\end{equation}\end{lemma}

\begin{proof} First, consider the distribution function of the product $\xi_1\cdot \xi_2$:
$$
F(z)=\textbf{P}\{\xi_1\cdot\xi_2<z\},\;z>0
$$
and the set of hyperbolae in the first quadrant
$$
\Gamma_z:=\{(x_1,x_2):x_1>0,x_2>0,x_1x_2=z\}.
$$
Given $z>0,$ denote $K\cap\Gamma_z=(x_1(z),x_2(z)).$ Select a small $h>0$ and obtain:
\begin{equation*}
\begin{split}
F(z+h)-F(z)&=\textbf{P}\big((\xi_1,\xi_2)\in \{(x_1,x_2):
x_2=\varphi(x_1),\;x_1(z)\leqslant x_1\leqslant x_1(z+h)\}\big)\\
&=\int_{x_1(z)}^{x_1(z+h)}f_1(x_1)dx_1=\int_{x_2(z)}^{x_2(z+h)}f_2(x_2)dx_2.
\end{split}
\end{equation*}
Applying the Mean Value Theorem yields:
$$
f_1(x_1^*)\left[x_1(z+h)-x_1(z)\right]=f_2(x_2^*)\left[x_2(z+h)-x_2(z)\right]
$$
for some $x_1^*$ and $x_2^*$
within the intervals of integration. As $h\to 0,$ one has: $x_1^*\to x_1(z)$ and, correspondingly,  $f_1(x_1^*)\to f_1(x_1(z)).$
Now, one can write that
$$
x_2(z+h)-x_2(z)=\left[\varphi^\prime(x_1(z))+o(1)\right]\left(x_1(z+h)-x_1(z)\right),
$$
while, on the other hand,
$$
x_2(z+h)-x_2(z)=\frac{z+h}{x_1(z+h)}-\frac{z}{x_1(z)}.
$$
Equating the expressions in the right-hand sides of the last two formulae leads to:
\begin{align*}
x_1(z+h)-x_1(z)=\frac{hx_1(z)}{z+\left[\varphi^\prime(x_1(z))+o(1)\right]x_1(z+h)x_1(z)}\,.
\end{align*}
Passing to limit as $h\rightarrow 0,$ one arrives at:
\begin{equation*}
\begin{split}
g(z,P)&=\lim_{h\to 0}\frac{F(z+h)-F(z)}{h}=\lim_{h\to 0} f_1(x_1^*)\frac{x_1(z)}{z+\left[\varphi^\prime(x_1(z))+o(1)\right] x_1(z+h)x_1(z)}\\
&=f_1(x_1(z))\frac{x_1(z)}{z+\varphi^\prime(x_1(z))x_1^2(z)}\,.
\end{split}
\end{equation*}
As a result, we obtain the density $g(z,P)$ of the product $\xi_1\cdot \xi_2$ in the form \eqref{pz}.
Similarly,  the expression for $g(z, P)$ in terms of $f_2$ can be derived.\end{proof}

\begin{example}\label{ex2}If $f_1=f_2=f,$ then $K=\{(x_1, x_2):x_1=x_2\},\;x_1(z)=\sqrt{z}$, and \eqref{pz} implies that
$$
g(z,\tilde P)=\frac{f(\sqrt{z})}{2\sqrt{z}},
$$
whence
$$
L_g(z)=\frac{1}{2}+\frac{1}{2}L_{f}(\sqrt{z}),
$$
demonstrating that $g$ satisfies Lin's condition whenever $f_1$ does.\end{example}

Finally, we present the proof of Theorem \ref{th1}.

\begin{proof}
To construct a singular probability distribution $P$ in $\mathbb{R}^2$ satisfying the conditions of Theorem \ref{th1}, first consider its construction inside a single rectangle $\Pi=[a,b]\times [\varphi (a), \varphi(b)].$
Denote by $Q$ the restriction of  distribution $\tilde P$ constructed in Lemma \ref{lem1} on rectangle $\Pi$.
Select a small $\delta >0$. Conditions on the value of $\delta$ will be specified later.

For each $x>0$ and $z>0,$ define the function $x=\rho(z)$ by the condition:
 \begin{equation}\label{rhoz}z=x\varphi(x).\end{equation}
  Obviously, $\rho$ is a continuous strictly increasing function on $(0,\infty).$

Then $a=\rho (s), b=\rho (t)$
for some $s<t$, and the curve $K$ has the following parametric representation in terms of $z$:
$$
\overrightarrow{\bf{x}}(z)= \left(x_1(z), x_2(z)\right)=\left(\rho(z),\varphi(\rho(z))\right).
$$
Choose $b^{\,\prime}=\rho(t-3\delta)<b$ close to $b$ and consider a non-negative function $ \tau\in C^\infty(0,+\infty)$ satisfying the conditions:
\begin{align}\label{tau}
\tau(z)=\left\{\begin{array}{ccl}0&\text{when}&z\notin [t-3\delta, t],\\
\varepsilon \sin^2(\nu z)&\text{when}&z\in [t-2\delta, t-\delta].\end{array}\right.
\end{align}
Here $\varepsilon$ and $\nu$ are positive real numbers satisfying certain restrictions which will be explained in the sequel.
Consider the arc
\begin{align*}\label{arcl}
l_\delta^{(1)}=\{(x_1,x_2)\in K: b^{\,\prime}\leqslant x_1\leqslant b\}
=\{\overrightarrow{\bf{x}}(z):t-3\delta\leqslant z\leqslant t\}
\end{align*}
and determine on the arc measure $\mu_1$ whose density with respect to $z$ equals
$\tau(z).$
Select $\varepsilon$ in \eqref{tau} so that $Q-\mu_1$ is a measure in $\Pi.$
Next, shift the
mass on $l_\delta^{(1)}$
with density $\tau(z)$
by $d=\varphi (b^{\,\prime}) - \varphi (a)$ units down so that the shifted arc $l_\delta^{(2)}=l_\delta^{(1)}-(0,d)$ has starting point at $(b^{\,\prime}, \varphi(a))$ and terminal point at $(b,\varphi (b)-d).$
Denote this measure $\mu_2$:
$$
\mu_2(B)=\mu_1(B+(0,d))
$$
where
$B$ is a Borel set in the plane and
$B+(0,d)= \{(x_1,x_2+d):(x_1,x_2)\in B\}.$
Arc $l_\delta^{(1)}$ has to be small enough
to ensure that
$\varphi(b)-\varphi(b^{\,\prime})<\varphi(b^{\,\prime})-\varphi(a)$ (this follows from the smallness of $\delta$).
The performed shifting defines a singular
measure $q_1$ in  rectangle $\Pi$ by:
\begin{equation*}
q_1(B)=(Q-\mu_1)(B)+\mu_2(B),
\end{equation*}
The construction of $q_1$ implies that the projections of $Q$ and $q_1$ on the $x_1$-axis are the same.
However, to guarantee the equal projections on both coordinate axes, the procedure has to be continued. Consider measure $\mu_2$ concentrated on $l_\delta^{(2)}$ and find $a^{\,\prime}>a$
from the condition $\varphi(a^{\,\prime})=\varphi(a)+(\varphi(b)-\varphi(b^{\,\prime})).$ When $\delta$ is small enough, it can be stated that $a^{\,\prime}<b^{\,\prime}.$ Let $\mu_3$ be the measure on arc
\begin{equation*}
l_\delta^{(3)}=\{(x_1,x_2):x_2=\varphi(x_1), a\leqslant x_1\leqslant a^{\,\prime}\}
\end{equation*}
possessing the same projection on the $x_2$-axis as $\mu_2.$
Measure $\mu_3$ is well-defined since function $\varphi$ is strictly increasing. It has to be noticed that, when $\varepsilon >0$ is small enough, the difference $Q-\mu_3$ is a measure on
$l_\delta^{(3)}.$ Finally, shift the mass defined by $\mu_3$ on $l_\delta^{(3)}$ as a whole by $d$ units up, that is, define measure $\mu_4(B)=\mu_3(B-(0,d))$ concentrated on $l_\delta^{(4)}=l_\delta^{(3)}+(0,d)$. As a result, we obtain a singular
measure on rectangle $\Pi=[a,b]\times [\varphi (a), \varphi(b)]$ given by:
\begin{equation*}\label{pnew}
\tilde{Q}=Q-\mu_1+\mu_2-\mu_3+\mu_4\,.
\end{equation*}
Here, $\delta$ and $\varepsilon$ are assumed to be sufficiently small to stipulate all the conditions mentioned above and the positivity of $\tilde{Q}.$ In addition, the projections of $\tilde{Q}$ on the coordinate axes coincide with the distributions of $\xi_1$ and $\xi_2$
on  $[a,b]$ and  $[\varphi(a),\varphi(b)]$ respectively,
that is, possess the given densities $f_1$ and $f_2$
on  $[a,b]$ and  $[\varphi(a),\varphi(b)]$. As before, denote by $g(z,\tilde Q)$ the density of $\xi_1\cdot\xi_2$ in the case when a joint distribution of $\xi_1$ and $\xi_2$ coincides with $\tilde Q$ in $\Pi.$
If $\delta >0$ is  small enough to ensure that arcs $l_\delta^{(2)}$ and $l_\delta^{(4)}$ have no intercepts with $\Gamma_z,$ then $g(z,\tilde Q)=g(z,\tilde P)$,  where $g(z,\tilde P)$ is given by \eqref{pz}.  To prove that  when $\nu$ is large enough, the derivative of $g(z,\tilde Q)$ with respect to $z$ can take  arbitrary large positive and large negative values, first notice that,
  for $z\in [t-2\delta, t-\delta]$ the following equality holds:
\begin{equation*}%\label{gz}
g(z,\tilde Q)=(f_1(x_1(z))-\varepsilon \sin^2(\nu z))\frac{x_1(z)}{z+\varphi^\prime(x_1(z))x_1^2(z)}.
\end{equation*} Applying \eqref{rhoz}, this can be restated as follows:
\begin{equation}\label{gzgz}
g(z,\tilde Q)=(f_1(\rho(z))-\varepsilon \sin^2(\nu z))\frac{\rho(z)}{z+\varphi^\prime(\rho(z))\rho^2(z)}.
\end{equation}
It has to be pointed out that $x_1=\rho(z)$ is a continuously differentiable function of $z$. To check this, recall that $\rho(z)\cdot \varphi(\rho(z))=z,$ whence
\begin{align*} \rho(z+h)-\rho(z)=\frac{z+h}{\varphi(\rho(z+h))}-\frac{z}{\varphi(\rho(z))}\\
=\frac{-z\left(\varphi(\rho(z+h))-\varphi(\rho(z))\right)+h\varphi(\rho(z))}{
\varphi(\rho(z+h))\varphi(\rho(z))}\\
=\frac{-z\varphi(\rho(z^*))\left[(\rho(z+h))-(\rho(z))\right]+h\varphi(\rho(z))}{
\varphi(\rho(z+h))\varphi(\rho(z))}.\end{align*}
The last equality is justified by The Mean Value Theorem since $\varphi$ is a differentiable function. By plain calculations one can obtain that:
\begin{equation*}\frac{\rho(z+h)-\rho(z)}{h}=\frac{\varphi(\rho(z))}{z\varphi(\rho(z^*))+\varphi(\rho(z+h))
\varphi(\rho(z))}.\end{equation*}
As $h\rightarrow 0,$ this leads to
\begin{equation*} \frac{d}{dz}\rho(z)=\frac{\varphi(\rho(z))}{z\varphi^\prime (\rho(z))+\varphi^2(\rho(z))}.\end{equation*}
Notice that the denominator in the right-hand side is strictly positive. Furthermore, with the help of \eqref{ivp} it can be concluded that $\varphi^{\prime\prime}$ exists, too.

Now, consider Lin's function for density \eqref{gzgz}:
\begin{equation*}L_g(z)=-z\frac{g^\prime(z;\tilde Q)}{g(z;\tilde Q)}.\end{equation*}
For $z\in [t-2\delta, t-\delta]$, the ratio $z/g(z;\tilde Q)$ is bounded, while the behaviour of
$g^\prime(z;\tilde Q)$ is governed by $\left(\varepsilon \sin^2(\nu z)\right)^\prime=\nu \varepsilon \sin (2\nu z),$ which oscillates rapidly taking arbitrary large in magnitude positive and negative values on $[t-2\delta, t-\delta]$ when $\nu$ is sufficiently large.

\bigskip Finally, in order to find a joint distribution of $\xi_1$ and $\xi_2,$ consider an infinite sequence of a disjoint
rectangles $\{\Pi_n\},\;\Pi_n=[a_n,b_n]\times [\varphi(a_n),\varphi(b_n)], a_n\to \infty.$
Denote by  $Q_n, n\in \mathbb{N},$ the restriction of measure $P$ on rectangle $\Pi_n$.
For each rectangle $\Pi_n,$ construct $\tilde{Q}_n$ as described above, making, at every step the derivatives $p^\prime(z_n^*)$ and $p^\prime(z_n^{**})$ as large in magnitude as necessary and of opposite signs. To complete the picture, set
\begin{equation*}
 P= \tilde P-\sum_{n=1}^{\infty}Q_n+ \sum_{n=1}^\infty\tilde Q_n.
\end{equation*}
It is clear that measure $ P$ satisfies all conditions stated in Theorem~\ref{th1}. \end{proof}

%\newpage

\end{document}